\documentclass[11pt]{amsart}
\usepackage{amssymb,amsmath,txfonts,amsthm}
\usepackage{hyperref}
\newtheorem{theorem}{Theorem}

\newtheorem*{lemma}{Lemma}
\newtheorem*{remark}{Remark}

\newtheorem*{cor}{Corollary}
\newtheorem*{conj}{Conjecture}

\def\XXint#1#2#3{{\setbox0=\hbox{$#1{#2#3}{\int}$}
  \vcenter{\hbox{$#2#3$}}\kern-.5\wd0}}

\author{Gang Liu}
\address{Department of Mathematics\\University of Minnesota\\Minneapolis, MN 55455}
\email{liuxx895@math.umn.edu}
\title[3-Manifolds with nonnegative Ricci curvature]{3-Manifolds with nonnegative Ricci curvature}
\date{}
\begin{document}
\begin{abstract}For a complete noncompact 3-manifold with nonnegative Ricci curvature, we prove that either it is diffeomorphic to $\mathbb{R}^3$ or the universal cover splits. This confirms Milnor's conjecture in dimension $3$.
\end{abstract}
\maketitle

\section{\bf{Introduction}}

Let $M$ be a complete manifold with nonnegative Ricci curvature, then it is a fundamental question in geometry to find the restriction of the topology on $M$. Recall in 2-dimensional case, Ricci curvature is the same as Gaussian curvature $K$. It is a well known result that if $K \geq 0$, the universal cover is either conformal to $\mathbb{S}^2$ or $\mathbb{C}$.

Let us consider 3-manifolds with nonnegative Ricci curvature.
By using the Ricci flow, Hamilton \cite{[H]} classified all compact 3-manifolds with nonnegative Ricci curvature. He proved that the universal cover is either diffeomorphic to $\mathbb{S}^3$ or $\mathbb{S}^2 \times \mathbb{R}$ or $\mathbb{R}^3$. In the latter two cases, the metric is a product on each factor $\mathbb{R}$.
For the noncompact case, there are some partial classification results. Anderson-Rodriguez \cite{[AR]} and Shi \cite{[Sh]} classified these manifolds by assuming the upper bound of the sectional curvature. Zhu \cite{[Z1]} proved that if the volume grows like $r^3$, then the manifold is contractible. Based on Schoen and Yau's work \cite{[SY1]}, Zhu \cite{[Z2]} also proved that if the Ricci curvature is quasi-positive, then the manifold is diffeomorphic to $\mathbb{R}^3$.

In late 1970s, Yau initiated a program of using minimal surfaces to study 3-manifolds. It turns out that this method is very powerful. For example, Schoen and Yau proved the famous positive mass conjecture \cite{[SY2]}\cite{[SY3]}. Meeks and Yau \cite{[MY1]}\cite{[MY2]} proved the loop theorem, sphere theorem and Dehn lemma together with the equivariant forms. In \cite{[SY1]}, Schoen and Yau proved that a complete noncompact 3-manifold with positive Ricci curvature is diffeomorphic to $\mathbb{R}^3$, they also announced the classification of complete noncompact 3-manifolds with nonnegative Ricci curvature.

In this note we classify complete noncompact 3-manifolds with nonnegative Ricci curvature in full generality. The proof is based on the minimal surface theory developed by Schoen and Fischer-Colbrie \cite{[FS]}, Schoen and Yau \cite{[SY1]} , Schoen \cite{[S]}. We will use the following theorem frequently.
\begin{theorem}\label{thm1}(Schoen-Yau\cite{[SY1]})
Let $M^3$ be a complete 3-manifold with nonnegative Ricci curvature. Let $\Sigma$ be a complete oriented stable minimal surface in $M$, then $\Sigma$ is totally geodesic, and the Ricci curvature of $M$ normal to $\Sigma$ vanishes at all points on $\Sigma$.
\end{theorem}

Below is our result:
 \begin{theorem}\label{thm2}
 Let $M^3$ be a complete noncompact 3-manifold with nonnegative Ricci curvature, then either $M^3$ is diffeomorphic to $\mathbb{R}^3$ or the universal cover of $M^3$ is isometric to a Riemann product $N^2 \times \mathbb{R}$ where $N^2$ is a complete 2-manifold with nonnegative sectional curvature.
 \end{theorem}

In \cite{[M]}, Milnor proposed the following conjecture:
 \begin{conj}
If a complete manifold has nonnegative Ricci curvature, then the fundamental group is finitely generated.
\end{conj}

\begin{cor}
Milnor's conjecture is true in dimension $3$.
 \end{cor}

\noindent\emph{Proof of the corollary.}

If $M$ is diffeomorphic to $\mathbb{R}^3$, then the conclusion is obvious.
Otherwise by theorem \ref{thm2},  $M$ has nonnegative sectional curvature. Hence the corollary follows from a result of Gromov \cite{[G]}.

\qed

\begin{center}
\bf  {\quad Acknowledgment} 
\end{center}

The author would like to thank Professors Richard Schoen, Jiaping Wang, Shing-Tung Yau for their interests in this note. 
He also thanks Chenxu He for informing him the paper \cite{[E]}.

\section{\bf{proof of the theorem}}

\noindent\emph{Proof of Theorem \ref{thm2}.}

We assume $M$ is not flat, otherwise the conclusion is obvious.

Let us review Schoen and Yau's argument in \cite{[SY1]}.
Assume $M$ is simply connected, if $\pi_2(M) \neq 0$,  according to Lemma 2 in \cite{[SY1]}, $M$ must have at least two ends. From Cheeger-Gromoll splitting theorem \cite{[CG]}, the universal cover splits. So we assume $\pi_2(M) = 0$. Therefore, the universal cover of $M$ is contractible. If $M$ is not simply connected, Schoen and Yau \cite{[SY1]} proved that $\pi_1(M)$ must have no torsion elements. Thus, after replacing $M$ by a suitable covering, we may assume that $\pi_1(M) = \mathbb{Z}$ and that $M$ is orientable. Let $\gamma$ be a Jordan curve representing the generator of the fundamental group of $M$. Consider an exaustion of $M$ by $\Omega_i$, where $\partial \Omega_i$ is a disjoint union of smooth 2-manifolds. We may assume that $\gamma$ lies in each $\Omega_i$. By Poincare duality for manifolds with boundary, there exists a oriented surface $\Sigma_i
\subset \Omega_i$ such that $\partial \Sigma_i \subset \partial \Omega_i$, moreover, the oriented intersection number of $\Sigma_i$ with $\gamma$ is 1. We would like to minimize the area among all surfaces which are in the same homology class as $\Sigma_i$ and with the same boundary as $\Sigma_i$. We can perturb the metric near $\partial \Omega_i$ such that the mean curvature is positive with respect to the outer normal vector. So there exists a minimizing surface for each $i$, which we still call $\Sigma_i$. For each $i$, the intersection of $\Sigma_i$ with $\gamma$ is nonempty. Therefore, a subsequence of $\Sigma_i$ converges to an oriented stable minimal surface $\Sigma$ in $M$.
 If the Ricci curvature is strictly positive on $M$, then this contradicts theorem \ref{thm1}.

Let us deal with the case when the Ricci curvature is nonnegative. For a fixed point $p \in M$, we may assume that $p$ does not lie on $\gamma$, otherwise we perturb $\gamma$ a little bit such that $p$ is not on $\gamma$.
According to the result in \cite{[E]} by Ehrlich, we can perturb the metric such that the Ricci curvature is strictly positive in a small annulus around $p$, while the metric remains the same outside the annulus(this means that inside the ball bounded by the annulus, the Ricci curvature might be negative). For reader's convenience, we give the details as follows:
According to the well-known formula, if $g(t) = e^{2tf}g_0$ and $|\nu|_{g(0)} = 1$, then
$$Ric^t(v, v) = e^{-2tf}(Ric(v, v) - t(n-2)\nabla^2 f(v, v) - t\Delta f + t^2(n-2)(v(f)^2-|\nabla f|^2))$$ where $n = dim(M) = 3$. Define $r$ to be the distance function to $p$. For a very small $R > 0$, consider the function
$\rho = R-r$ for $\frac{R}{2}< r < R$. Then we extend $\rho$ to be a positive smooth function for $0 \leq r < \frac{R}{2}$.
Define $f = -\rho^5$, for $|v| = 1$,
$$Ric^t(v, v)=e^{2t\rho^5}(Ric(v, v) + t(n-2)\nabla^2 (\rho^5)(v, v) + t\Delta (\rho^5) + t^2(n-2)(v(\rho^5)^2-|\nabla \rho^5|^2)).$$
Now $\nabla^2 (\rho^5)(v, v) = 20\rho^3v(\rho)^2+ 5\rho^4\nabla^2(\rho)(v, v)$, therefore,
\begin{equation}
Ric^t(v, v) \geq e^{2t\rho^5}(Ric(v, v) + 20t\rho^3+ 5t\rho^4(\Delta \rho+ (n-2)\nabla^2 (\rho)(v, v)) -25(n-2)t^2\rho^8).
\end{equation}
From now on, we restrict $r$ such that $\lambda R < r < R$, where $\lambda > \frac{1}{2}$ is to be determined.
Using the fact that near $p$, the manifold is almost Euclidean, for small $R$,
we have $$|\Delta \rho + (n-2)\nabla^2 \rho(v, v)|  \leq \frac{9(2n-3)}{8(R-\rho)}.$$
We plug this in (1). So for all small $t$, $g(t)$ have strictly positive Ricci curvature in an annulus $B_p(R)\backslash B_p(\lambda R)$ for $\lambda = \frac{7}{8}$. The metric remains the same outside $B_p(R)$. The deformation is $C^4$ continuous with respect to the metric and $C^{\infty}$ with respect to $t$.

 We apply this perturbation finitely many times so that the Ricci curvature is positive on $\gamma$(each time we perturb the metric a little bit around a point) and that the Ricci curvature is nonnegative except a small neighborhood of $p$. Then we can minimize the area as before. This will yield a complete stable minimal surface $\Sigma$. Now the claim is that $\Sigma$ must pass through the small neighborhood of $p$. If this is not true, then on $\Sigma$,  the Ricci curvature is nonnegative, the normal Ricci curvature is strictly positive somewhere on $\gamma$. This contradicts theorem \ref{thm1}.

Using $t$ to denote the deformation parameter, we shrink the size of the neighborhood of $p$ where the Ricci curvature might be negative. So we get a sequence of metrics on $M$ and for each metric, a stable minimal surface passing through a small neighborhood of $p$. We may let $t\to 0$ sufficiently fast so that these metrics are converging to the initial metric in $C^4$ sense. Taking the limit for a subsequence of these complete minimal surfaces, we obtain a complete oriented stable minimal surface passing through $p$, with the initial metric. According to theorem \ref{thm1}, this surface is totally geodesic with vanishing normal Ricci curvature.

Since the manifold is not flat, there exists a neighborhood $U$ such that the scalar curvature is strictly positive in $U$. Consider a point $p \in U$ and a sequence of points $p_i \to p$, where all $p_i \in U$. Through each $p_i$, there exists a complete totally geodesic surface $H_i$.  So a subsequence of $H_i$ converges to a complete totally geodesic surface $H$ through $p$. We assume that the normal vector of $H_i$ at $p_i$ converges to the normal vector of $H$ at $p$. We can choose $p_j$ so that for any $j > i$, $p_j$ does not lie on $H_i$. Therefore, for all large $i$, $H_i$ does not coincide with $H$.

By the assumption of $U$, $H_i$ and $H$ are not flat. They have nonnegative sectional curvature, so they are conformal to $\mathbb{C}$. The normal bundle is trivial. We denote the unit normal vector of $H$ by $N$. For any $x \in H$, when $k$ is very large, we shall construct a piece $\Sigma_k \subset H_k$. For a shortest geodesic on $H$ connecting $p$ and $x$, we assume $x = exp_p(v)$ where $v \in T_{p}{H}$. If the geodesic is not unique, then we just choose one. We parallel transport the vector $v$ along the shortest geodesic connecting $p$ and $p_k$ to obtain a tangent vector $u_k$ at $p_k$. Then we project $u_k$ to $T_{p_k}(H_k)$ to get $v_k \in T_{p_k}(H_k)$. Define a point $x_k = exp_{p_k}v_k$. Since we may have multiple choices of $v$, $x_k$ may be different. However, when $k$ is very large, these $x_k$ are close to $x$, since $p_k \to p$ and the normal vector of $H_k$ at $p_k$ is converging to the normal vector of $H$ at $p$. Moreover, these $x_k$ belong to the same piece of $H_k$, i.e, the $H_k$ distances between them are very small, since $H_k$ and $H$ are simply connected. Let $r = \frac{1}{10}inj_M(x)$ where $inj_M(x)$ denotes the injective radius of $M$ at $x$. Define $\Sigma_k = B_{H_k}(x_k, r)$. From the construction of $x_k$, for $k$ large, the normal vector of $H$ at $x$ and the normal vector of $H_k$ at $x_k$ are close in the obvious sense, as the normal vectors of $H$ and $H_k$ are parallel along each surfaces. Since $x_k$ is very close to $x$, $inj_M(x_k) \geq \frac{1}{2}inj_M(x) \geq r$. Therefore $dist_M(\partial B_{H_k}(x_k, r), x) \geq r - dist_M(x_k, x) > 5 dist_M(x, x_k)$ for $k$ large. Thus if $l$ is the normalized shortest geodesic connecting $x$ and $\Sigma_k$, $l$ will intersect the inner part of $\Sigma_k$, say at the point $\overline x_k$. Triangle inequality implies that $dis_{H_k}(x_k, \overline x_k) \leq 2 dis_M(x, x_k)$. Therefore, the unit normal vector of $H$ at $x$ and the unit normal vector of $H_k$ at $\overline x_k$ are close in the obvious sense.

Denote the initial tangent vector of $l$ at $x$ by $e$. The oriented distance is defined by $d_k(x) = dist_M(x, \Sigma_k)Sign(\langle e, N\rangle)$ for $x \in H$. The function $Sign(t) = 1$ when $t > 0$; $Sign(t) = -1$ when $t < 0$; $Sign(t) = 0$ when $t = 0$. For any $x \in H$, $d_k(x)$ is well defined and smooth for $k$ sufficiently large. Via the second variation of arc length, there is a nice pinching estimate for the Hessian of $d_k(x)$ when $d_k(x)$ is very small, namely, $$-d_k(x)(R_{NijN}+Sign(d_k(x))\epsilon(k, x)) \leq (d_k(x))_{ij} \leq -d_k(x)(R_{NijN}-Sign(d_k(x))\epsilon(k, x))$$
where $\lim\limits_{k \to \infty}\epsilon(k, x) = 0$ and the convergence is uniform for any compact set of $H$. In the above estimate, we have used the fact that for $k$ large, the normal direction of $H_k$ at $\overline x_k$ and the normal direction of $H$ at $x$ are close in the obvious sense.
Since $d_k$ does not vanish identically, after a suitable rescaling, a subsequence converges to a nonzero function $f$ when $k \to \infty$. Then $f$ satisfies
\begin{equation} f_{ij} + fR_{NijN} = 0
 \end{equation}
where $f_{ij}$ is the Hessian of $f$ on $H$ with the induced metric. Moreover, $\Delta f = 0$ since the normal Ricci curvature vanishes identically.

\begin{remark}
We use the rescaled distance function to approximate the variational vector field on $H$. If the surfaces $H_k$ and $H$ are properly embedded, then we can simply define $d_k(x) = dist_M(x, H_k)Sign(\langle e, N\rangle)$. We define the function $d_k(x)$ as in last paragraph because in the final part of the paper, when we try to show that $M$ is simply connected at infinity, we obtain stable minimal surfaces which could be immersed and improper.
\end{remark}

\begin{lemma} $f \equiv Constant$.
\end{lemma}
\begin{proof}
First, $H$ is conformal to $\mathbb{C}$, since it is not flat and the Gaussian curvature is nonnegative.
We may assume $f$ changes sign, otherwise from the Liouville property for positive harmonic functions on $H$, $f$ is constant. We observe that the vanishing points of $f$ consists of the geodesics on $H$, since $\nabla f$ is parallel along the vanishing points of $f$(the hessian of $f$ vanishes when $f$ vanishes, see (2)). Moreover, these geodesics do not intersect, otherwise $\nabla f = 0$ along one geodesic. Combining this with (2), we find $f \equiv 0$. This is a contradiction.

Now suppose the zero set of $f$ contains at least 2 distinct geodesics. Let us call them $L_1, L_2$. We claim that $L_1, L_2$ are proper on $H$. The reason is this: we can write $f$ as the real part of a holomorphic function $h = f + ig$, since $f$ is harmonic. By Cauchy-Riemann relation, along the vanishing set of $f$, $g$ is strictly monotonic, $|\nabla g|$ is constant along $L_1$ and $L_2$(since $|\nabla f|$ is constant on each of these two geodesics). But in a compact set of $H$, $|h|$ is bounded, therefore, $L_1$, $L_2$ are properly embedded on $H$. Consider the function $d(x) = dist_H(x, L_2)$ for $x \in L_1$. From the Hessian comparison, we can show that $d'' \leq 0$. Since $L_1$ and $L_2$ never intersect, $d(x) \equiv d_0$. Using the Hessian comparison again, we find the metric to be flat in the domain $\Omega$ bounded by $L_1$ and $L_2$ on $H$. therefore the scalar curvature of the ambient space vanishes on $\Omega$. Considering (2), we find that $f$ is linear on $\Omega$. However, the vanishing points of $f$ have two components, this is a contradiction.

Thus the vanishing points of $f$ consist of one geodesic. By the monotonicity of $g$, for any $t \in \mathbb{R}$, there exists exactly one solution to the equation $h(z) = (0, t) \in \mathbb{C}$.  By big Picard theorem for entire functions, infinity can not be an essential singularity for the entire function $h$, since $h$ can take each value $(0, t)$ only once.
Therefore, $h$ is a polynomial. Using again that there exists exactly one solution to the equation $h(z) = (0, t) \in \mathbb{C}$, we find $h$ to be a linear function. After some conformal transformation, we may assume $f = x$ on the complex plane.
Suppose the metric on $H$ is given by $ds^2 = e^{2\rho}(dx^2 + dy^2)$ using Cartisian coordinate on $\mathbb{C}$.

Let $e_1 = \frac{\partial}{\partial x}, e_2 = \frac{\partial}{\partial y}$,
then $$\langle \nabla_{e_1}e_1, e_1\rangle = e^{2\rho}\rho_1, \langle \nabla_{e_1}e_1, e_2\rangle = -\langle \nabla_{e_2}e_1, e_1\rangle = -e^{2\rho}\rho_2.$$

Therefore $$\nabla_{e_1}e_1 = \rho_1e_1 - \rho_2e_2.$$ Similarly $$\nabla_{e_1}e_2 = \nabla_{e_2}e_1 = \rho_2e_1 + \rho_1e_2, \nabla_{e_2}e_2 = \rho_2e_2 - \rho_1e_1.$$
So the Hessian of $f$ is given by
$$f_{11} = 0 - (\nabla_{e_1}e_1)f = -\rho_1, f_{12} = 0 - (\nabla_{e_1}e_2)f = -\rho_2, f_{22} = 0 - (\nabla_{e_2}e_2)f = \rho_1.$$
Let us write (2) as $f_{ij} + f\tau_{ij} = 0$.
Therefore, the norm of the tensor $\tau$ is $$|\tau_{ij}| = \frac{\sqrt{2}|\nabla_E \rho|}{|x|e^{2\rho}}$$(here $\nabla_E, \Delta_E$ denotes the gradient and the Laplacian with respect to the standard metric on $\mathbb{C}$).
Since the Ricci curvature of the ambient manifold is nonnegative and that the normal Ricci curvature vanishes, $|\tau_{ij}| \leq \sqrt{2}K$ where $K = -\frac{\Delta_E \rho}{e^{2\rho}}$ is the Gaussian curvature on the surface.
Therefore $$\frac{|\nabla_E \rho|}{|x|} \leq -\Delta_E \rho.$$
Let $h = -\rho$, so $$\Delta_E h \geq \frac{|\nabla_E h|}{|x|} \geq \frac{|\nabla_E h|}{r}$$ where $r^2 = x^2 + y^2$. By Cohn-Vossen inequality, $\int K ds^2 \leq 2\pi$. Therefore, $$\int \frac{|\nabla_E h|}{|x|}dxdy \leq \int \Delta_E h dxdy< \infty.$$

Define $$g(t) = \int_{B(t)} \frac{|\nabla_E h|}{r}dxdy$$ where $B(t)$ is the Euclidean disk centered at the origin with radius $t$. We have
$$t\int_{\partial B(t)}\frac{|\nabla_E h|}{r}dl \geq \int_{B(t)} \Delta_E h dxdy \geq \int_{B(t)} \frac{|\nabla_E h|}{r}dxdy.$$ That is to say,  $$tg' \geq g.$$ Solving this inequality, combining with the condition that $g$ is bounded, we find that $$g \equiv 0.$$
Therefore $H$ is flat. But this contradicts the assumption that $H$ is not flat.
Thus the lemma is proved.

\end{proof}

We plug this result in (2). It turns out that $R_{iNNj} = 0$ on $H$. So in fact the rank of the Ricci curvature is 2 at $p$. Therefore, through each point close to $p$, there is a unique totally geodesic surface. From linear algebra, we see these surfaces vary smoothly. By the calculus of variation, the variational vector field of each surface satisfies equation (2). According to the lemma, after a reparametrization, we may assume the variational vector fields of these surfaces are given by $\nu = N$. We call these surfaces $\Sigma_t$, $-\epsilon < t < \epsilon$. Given a point $x \in \Sigma_t$, if $X \in T_x\Sigma_t$, then $\nabla_XN=0$, as $\Sigma_t$ is totally geodesic.  Since $N = \nu$, we may extend $X$ in a small neighborhood of $x$ in $M$ such that $X \in T\Sigma$ and $[X, N] = 0$. We have $<\nabla_NN, X> = -<\nabla_NX, N>= -<\nabla_XN, N> = 0$. Since $X \in T_x\Sigma_t$ is arbitrary, $\nabla_NN = 0$. Thus the unit normal vector of these surfaces is parallel and $\Sigma_t$ are all isometric to $\Sigma_0$ via the integral curve of the variational vector field. Let $I$ be the maximal connected interval of $t$ such that there exists a local isometry $F: \Sigma \times I \to M$ with $F(\Sigma, 0) = \Sigma_0$.  From the definition of $I$, it is easy to see that $I$ is closed. Let $c(t)$ denote the integral curve of the normal vector field $N$ such that $c(0) = p$. Then for any $t \in I$, the scalar curvature at $c(t)$ are the same, since $F$ is a local isometry. $I$ is open, since for any $t\in I$, the scalar curvature at $c(t)$ is positive, we can extend $I$ a little bit more at the end points. Therefore we have a local isometry $F: \Sigma \times \mathbb{R} \to M$, which means that the universal cover of $M$ splits.

Now assume that $M$ is contractible. To prove that $M$ is diffeomorphic to $\mathbb{R}^3$, from a topological result by Stallings \cite{[St]}, it suffices to prove that $M$ is simply connected at infinity and irreducible. Suppose $M$ is not simply connected at infinity, this means that there exists a sequence of closed curves $\sigma_i$ tending to infinity such that for any immersed disk $D_i$ with $\partial D_i = \sigma_i$, $D_i \cap K \neq \Phi$ where $K$ is a fixed compact set of $M$. We may assume these disks are area minimizing, by the compactness and regularity result in Theorem 3 of \cite{[S]}, a subsequence of $D_i$ converges to a complete stable minimal surface which could be immersed and improper.

We can apply the argument as before. For reader's convenience, we give some details here.
Given a point $p\in M$, we perturb the metric such that $Ric > 0$ in $K\backslash B_p(r)$ and $Ric \geq 0$ in $M\backslash B_p(r)$. Then for the perturbed metric, we have a complete immersed(not necessarily proper) stable minimal surface $\Sigma_i$ which intersects $K$, thus intersects $B_p(r)$ at some $p_i$. The surfaces $(\Sigma_i, p_i)$ have uniform regularity in any compact set in $M$. When the perturbation is smaller and smaller, a subsequence of $(\Sigma_i, p_i)$ converges to a stable minimal surface $(\Sigma, p)$. According to theorem \ref{thm1}, $\Sigma$ is totally geodesic and the normal Ricci curvature vanishes.
Then we can use arguments in page $4, 5$ and $6$ to show that $M$ splits, which contradicts that $M$ is not simply connected at infinity.

To prove that $M$ is irreducible, we can invoke the solution of Poincare conjecture by Perelman \cite{[P1]}\cite{[P2]}\cite{[P3]}.
Therefore $M$ is diffeomorphic to $\mathbb{R}^3$. This completes the proof of theorem \ref{thm2}.

\qed

   \end{document}